\documentclass{article}

\usepackage[left=25mm, right=25mm, top=20mm, bottom=20mm]{geometry}

\usepackage{amsmath}
\usepackage{amssymb}
\usepackage{amsfonts}
\usepackage{setspace}
\usepackage{array}
\usepackage{amsthm}
\usepackage{graphicx}
\usepackage{amstext}
\usepackage{mathrsfs}
\usepackage{cases}
\usepackage{nicefrac}
\usepackage{caption}
\usepackage{subcaption}

\theoremstyle{definition}
\newtheorem{theorem}{Theorem}[section] 
\newtheorem{example}[theorem]{Example}
\newtheorem{definition}[theorem]{Definition}
\newtheorem{proposition}[theorem]{Proposition}
\newtheorem{corollary}[theorem]{Corollary}
\newtheorem{lemma}[theorem]{Lemma}

\theoremstyle{remark}
\newtheorem*{remark}{Remark}

\newcommand{\B}{\mathscr{B}}
\newcommand{\G}{\mathscr{G}}

\newcommand{\Z}{\mathbb{Z}}

\newcommand{\R}{\mathbb{R}}
\newcommand{\Q}{\mathbb{Q}}

\newcommand{\N}{\mathbb{N}}

\everymath{\displaystyle}
\allowdisplaybreaks

\usepackage{authblk}

\makeatletter

\title{An analogue of Solomyak's theorem for periodic Cantor real expansions in alternate bases}

\author{Jonathan Caalim \thanks{jcaalim@math.upd.edu.ph} }

\author{Nathaniel Nollen \thanks{nnollen@math.upd.edu.ph}}

\affil{\small{University of the Philippines - Diliman, Quezon City, Philippines 1101}}
\date{}

\begin{document}

\maketitle
\begin{abstract}
    
In this paper, we consider the positional numeration system, called the Cantor real expansion, on the unit interval $[\gamma, \gamma+1]$, where $\gamma \in \mathbb{R}$, with respect to an alternate base (i.e., a base which is a purely periodic sequence of real numbers). In particular, we study the case where the expansion of $\gamma+1$ is periodic. Under certain assumptions, the base satisfies algebraic properties. We compute the bounds for the norms of the nontrivial Galois conjugates associated with the base; thereby, extending the results of Solomyak on the classical beta expansions.
\\

\end{abstract}

\section{Introduction}

Let $Q=(q_n)_{n \ge 1}$ be a sequence of integers $q_n \ge 2$. 
The $Q$-Cantor series expansion \cite{Cantor} of a real number $x$ is the unique sum of the form
$$x=E_0 + \sum\limits_{n \ge 1} \frac{E_n}{ \Pi_{j=1}^n{q_{j}}} $$
where $E_0 = \lfloor x \rfloor$ and $E_n \in \{0,1, \dots, q_n-1\}$ for all $n \ge 1$ such that 
$E_n \ne q_n-1$ infinitely many number of times. This positional numeration system may be generalized by taking the radix $Q$ to be a real sequence (see \cite{Dem2020,charlier2021expansions}). 

    Let $B=(\beta_1,\beta_2, \dots) $ be a sequence of real numbers $\beta_i$. Let  $\gamma \in \mathbb{R}$. For $j\in \mathbb{N}$, we define self-maps $f_j$ and $T_j$ on $[\gamma,\gamma+1)$ by    
    \begin{align*}
f_j(x) &= \beta_j x- \left\lfloor \beta_j x - \gamma \right\rfloor\\
&=\{\beta_j x - \gamma\} +\gamma
    \end{align*}
and 
\[
 T^{m}(x)=  f_m \circ \dots \circ f_3 \circ f_2 \circ f_1 (x). 
 \]
Putting
\[a_m (x) =  \left\lfloor \beta_m  T^{m-1} (x)-\gamma \right\rfloor,\]
we get
\[T^m (x) = \beta_m  T^{m-1} (x) - a_m (x).\]
We write $T_{B}$ or $T_{B,\gamma}$ instead of $T$ whenever we need to specify $B$ or $\gamma$. 
Let
$B[i,j] := \Pi_{m=i}^{j} {\beta_{m}}$ and $B[j]:=B[1, j]$ with $B[0]:=1$.  
If $\lim_{m \rightarrow \infty} |B[m]| =\infty$, a real number  $x\in [\gamma,\gamma+1)$ has the Cantor real expansion (or, simply, $B$-expansion)
\[d(B;x):=(a_1(x), a_2(x), \dots)\]
and 
\[x=\sum_{m=1}^{\infty}\frac{a_m(x)}{B[m]}=:(a_1, a_2, \dots)_B.\]
If $B=(\overline{\beta_1, \beta_2, \dots, \beta_N})$ is periodic, we also refer to $d(B;x)$
as the $(\beta_1,\dots, \beta_N)$-expansion of $x$. We call the base $B$ an alternate base and write
$d(B;x) = d(\beta_1,\dots, \beta_N; x)$.

When $B$ is an alternate base, Charlier et.al. \cite{charlier2023spectrum} provided a necessary condition for the expansion of 1 (here, $\gamma =0$) to be eventually periodic. This result extends the work of Parry  \cite{Parry1960} on the classical \textit{beta expansion}  (which was introduced by R{\'{e}}nyi in \cite{Rnyi1957}).
In particular, Parry showed that if the beta expansion of 1 is eventually periodic (i.e. the base $\beta$ is a Parry number), then $\beta$ is an algebraic.

In fact, Parry also showed if $\lambda\neq\beta$ is a Galois conjugate of $\beta$, then $|\lambda|\leq 2$. 
Solomyak \cite{Solomyak1994} improved this to $|\lambda|\le (1+\sqrt{5})/2$ and proved that this bound is sharp. 

An analog of the result of Solomyak exists for other numeration systems related to beta expansion.
In \cite[Theorem 3.7]{Thompson2017},  Thompson gave an analog of the result of Solomyak for the class of so-called `generalized $\beta$-maps'. Suzuki considered the same problem but restricted to two subclasses of these $\beta$-maps which include a version of negative $\beta$-transformations (see  \cite[Corollary 2.4, Theorem 2.5]{suzuki2021galois}). Negative beta transformations were introduced by Ito and Sadahiro in \cite{Sadahiro}
 to allow negative real numbers as radices for beta expansions. 
        
The goal of this article to provide an analog of the result of Solomyak for Cantor real expansion when $B$ is an alternate base.

\section{Periodic $B$-expansions} \label{periodic}

The $m$th digit $a_m(x)$ of the $B$-expansion $d(B;x)$ is an element of $\mathcal{A}(\beta_m)$
where 
\[\mathcal{A}(\beta) := \begin{cases} 
   [u_\beta, v_\beta) \cap \mathbb{Z}   & \text{if $\beta>0$ and $\beta +\gamma(\beta-1) \in \mathbb{Z}$}\\
   [u_\beta, v_\beta] \cap \mathbb{Z}   & \text{otherwise}
\end{cases}\]
and 
\[u_\beta:=\min\{ \left\lfloor \beta\gamma  -\gamma \right\rfloor, \left\lfloor \beta(\gamma+1)  -\gamma \right\rfloor \},\]
\[v_\beta:=\max\{ \left\lfloor \beta\gamma  -\gamma \right\rfloor, \left\lfloor \beta(\gamma+1)  -\gamma \right\rfloor \}\]
for $\beta= \beta_m$.
Thus,
\[d(B;x)\in \mathcal{A}(B):= \mathcal{A}(\beta_1) \times \mathcal{A}(\beta_2) \times \cdots .\]

As has been done in \cite{Parry1960, Sadahiro}, 
we define the $B$-expansion of $\gamma+1$ in two ways:
\begin{enumerate}
\item Define $T(\gamma+1):\beta_1(\gamma+1)-\lfloor \beta_1 (\gamma+1)-\gamma \rfloor \in [\gamma,\gamma+1)$ and obtain the expansion $d(B;\gamma+1)$ as usual. If $d(B;\gamma+1)=(c_1,c_2,\dots)$, then the first digit $c_1\in [u_{\beta_1},v_{\beta_1}] \cap \mathbb{Z}$ and
$c_j\in \mathcal{A}(\beta_j)$ for $j\ge 2$.

\item Define $d^*(B;\gamma+1) := \lim_{\varepsilon \rightarrow 0^+}d(B;\gamma+1-\varepsilon).$ Note that $d^*(B;\gamma+1) \in \mathcal{A}(B)$.
\end{enumerate}
If $d(B;\gamma+1)$ or $d^*(B;\gamma+1)$ is equal to $(d_1,d_2,\dots)$, then
$\displaystyle{ \gamma+1 = \sum_{i=1}^{\infty} \frac{d_i} {B[i]} }$.

Let $B=(\overline{\beta_{1},\beta_{2},\dots, \beta_{n}})$ be an alternate base of length $n$.
Suppose that the $B$-expansion of $x \in [\gamma, \gamma+1)$ is eventually periodic 
with $k$ and $m$ as the lengths of the pre-periodic and periodic parts, respectively; that is
\[x=(a_1,a_2, \dots, a_k, \overline{a_{k+1},a_{k+2}, \dots a_{k+m}})_B.\]
Without loss of generality, we may assume that both $k$ and $m$ are divisible by $n$ since
\[(a_1,a_2, \dots, a_k, \overline{a_{k+1},a_{k+2}, \dots a_{k+m}})=(a_1,a_2,\dots,a_{nk}, \overline{a_{nk+1},\dots,a_{nk+nm}}).\]
Let
\begin{align*}
y_j:&= \left(\frac{a_{n(j-1)+1}}{B[1]}+ \frac{a_{n(j-1)+2}}{B[2]}+\dots +  \frac{a_{nj}}{B[n]}\right)B[n]\\
&=\sum_{i=0}^{n-1}a_{nj-i}B[n+1-i,n],
\end{align*}
with $B[n+1,n]=1$.
Then
    \begin{align*}
            x&=\sum_{j=1}^{k}\dfrac{y_j}{(B[n])^j}+
            \left(\sum_{j=k+1}^{k+m}\dfrac{y_j}{(B[n])^j}
            \right)
            \left(\sum_{j=0}^\infty \dfrac{1}{(B[n])^{mj}}\right)\\
            &= \sum_{j=1}^{k}\dfrac{y_j}{(B[n])^j}+
            \left(\sum_{j=k+1}^{k+m}\dfrac{y_j}{(B[n])^j}
            \right)
            \left(\dfrac{(B[n])^{m}}{(B[n])^{m}-1}\right).
        \end{align*}
Note that above infinite geometric series is convergent since $\lim_{j\to\infty}|B[j]|=\infty$, and consequently, $|B[n]|>1$.      
Hence, we have   
\[
           x \left((B[n])^m-1\right) = \sum_{j=1}^{k+m} y_j(B[n])^{m-j} - \sum_{j=1}^{k}\dfrac{y_j}{(B[n])^j}.
\]
Therefore,
\begin{equation}
           x \left((B[n])^{m+k}-(B[n])^{k}\right) = \sum_{j=1}^{m+k} y_j(B[n])^{m+k-j} - \sum_{j=1}^{k} y_j(B[n])^{k-j}.\label{eq:1}
\end{equation}

Hence, we have the following result on the algebraicity of $\beta_i$ over some number field under certain conditions. See also \cite[Theorems 14 and 19]{charlier2023spectrum} for a parallel statement.

\begin{proposition} \label{algebraic}
Let $K_i:=\Q(\beta_1,\dots,\beta_{i-1},\beta_{i+1},\beta_n)$. Let $x\in K_i\cap [\gamma, \gamma+1]$. If $d(B;x)$ or $d^*(B;x)$ is eventually periodic such that $T^{i-1}(x)\neq0$, then $\beta_i$ is algebraic over $K_i$. 
\end{proposition}\label{algebraicconstruction}

\begin{proof}
Note that if $\beta_i=\beta_j$ for some $j\neq i$, then $\beta_i\in K_i$ and so, $\beta_i$ is trivially algebraic over $K_i$.
Now, suppose $\beta_i \neq \beta_j$ for any $i\neq j$. Apply the mapping $\beta_i\rightarrow X$ to Eq. \eqref{eq:1}. It suffices to show that the resulting polynomial in $X$ is not constantly zero. The leading coefficient of the LHS is 
\[\frac{xB[n]^{m+k}}{(\beta_i)^{m+k}},\]
while
the leading coefficient of the RHS is
\[ \dfrac{(a_1B[2,n]+\cdots+a_{i-1}B[i,n])B[n]^{m+k-1}}{(\beta_i)^{m+k}} = \left(x-\dfrac{T^{i-1}(x)}{B[i-1]}\right)\frac{B[n]^{m+k}}{(\beta_i)^{m+k}}.\]
These two coefficients are not equal if and only if 
\[x\neq x-\dfrac{T^{i-1}(x)}{B[i-1]}. \qedhere\]

\end{proof}

\begin{remark}
In particular, $\beta_1$ is algebraic over $\Q(\beta_2,\dots,\beta_n)$ if $d(B;\gamma)$ is eventually periodic and $\gamma \neq 0$; or if
one of $d(B;\gamma+1)$ and $d^*(B;\gamma+1)$ is eventually periodic and $\gamma+1\neq 0$.
\end{remark}

We now begin our exposition on the main results of this paper: Theorem \ref{bounds} and Corollary \ref{boundforlambda}. The discussion here follows the proof presented by Solomyak in \cite{Solomyak1994}.
The key difference of this paper from \cite{Solomyak1994} is analysis of the irreducibility of certain polynomials over a number field $\mathbb{Q}(\beta)$ used in the proof of Theorem \ref{equality}.

Suppose \(x \in [\gamma, \gamma+1]\) such that $y:=T^{i-1}(x)\neq0$.  Suppose that the $B$-expansion of $x$ is eventually periodic.
Then, the $B_i$-expansion of $y$ is eventually periodic where $B_i=(\overline{\beta_i,\dots,\beta_{n},\beta_1,\dots,\beta_{i-1}})=:(\overline{\beta'_1,\dots,\beta'_n})$. 

Suppose the $B_i$-expansion of $y$ is
\[(a_1,a_2,\dots,a_{nm},\overline{a_{nm+1},\dots,a_{nm+qn}}).\]
Here, we adapt the notation $y_j$ with the underlying base $B_i$.
We also use $B'[j,k]:=\prod_{m=j}^k\beta_m'$ with $B'[j]:=B'[1,j]$.

Suppose
$y=\sum_{j=1}^{\infty}\frac{y_j}{(B'[n])^j}$. For brevity, we write
$T^j$ instead of $T^j(y)$.
Note that $a_{nj-k}=\beta_{nj-k}'T^{nj-k-1}-T^{nj-k}$.
Then
\begin{align*}
y_j&=\sum_{k=0}^{n-1}a_{nj-k}B'[n+1-k,n]\\
 &=\sum_{k=0}^{n-1}(\beta_{nj-k}'T^{nj-k-1}-T^{nj-k})B'[n+1-k,n]\\
 &=T^{nj-n}B'[n]-T^{nj}.
\end{align*}

Recall that $\beta_i$ is algebraic over $K_i:=\Q(\beta_1,\dots,\beta_{i-1},\beta_{i+1},\dots,\beta_n$). Let $\lambda$ be a nontrivial conjugate 
of $\beta_i$ over $K_i$. Suppose that $B'[2,n]$ is fixed by the conjugate map $K_i(\beta_i) \longrightarrow K_i(\lambda)$ given by $\beta_i\to \lambda$.
Applying the map $\beta_i \rightarrow \lambda$, we get

\begin{align*}
y&=\sum_{j=1}^{\infty}\frac{y_j}{(\lambda B'[2,n])^j}\\
& =\sum_{j=1}^{\infty} \frac{T^{nj-n}B'[n]}{ ( \lambda B'[2,n])^j }-\sum_{j=1}^{\infty}\frac{T^{nj}}{(\lambda B'[2,n])^j}\\
& = \frac{\beta_i}{\lambda}\sum_{j=1}^{\infty}\frac{T^{nj}}{(\lambda B'[2,n])^j}-\sum_{j=1}^{\infty}\frac{T^{nj}}{(\lambda B'[2,n])^j}
+\frac{y \beta_i}{\lambda}.
\end{align*}

Therefore,
\[\left( \frac{\beta_i}{ \lambda}-1\right) \left( y+ \sum_{j=1}^{\infty}\frac{T^{nj}}{(\lambda B'[2,n])^j} \right)=0.\]
Since $\lambda \neq \beta_i$, then 
\begin{equation}
y+ \sum_{j=1}^{\infty}\frac{T^{nj}}{(\lambda B'[2,n])^j} = 0. \label{eq:2}
\end{equation}
In other words,
$(\lambda B'[2,n])^{-1}$ satisfies an analytic function
\[f(z)=y+\sum_{j=1}^\infty d_jz^j,\]
where $d_j\in[\gamma,\gamma+1]$.
Let $M_i=B'[2,n]=B[n]/\beta_i$. 
Let $z=( \lambda M_i)^{-1}$ and suppose $|z|< 1$. Then

\begin{eqnarray*}
    0 & = & y+\sum_{j=1}^\infty\left(\gamma+\dfrac{1}{2}\right)z^j+\sum_{j=1}^\infty \left(d_j-\gamma-\dfrac{1}{2}\right)z^j\\
    & = &\dfrac{2y-(2y-2\gamma-1)z}{2(1-z)}+\sum_{j=1}^\infty \left(d_j-\gamma-\dfrac{1}{2}\right)z^j.
\end{eqnarray*}
Hence, 
\[\left|\dfrac{2y-(2y-2\gamma-1)z}{2(1-z)}\right|=\left|\sum_{j=1}^\infty \left(d_j-\gamma-\dfrac{1}{2}\right)z^j\right|\leq \sum_{j=1}^\infty \dfrac{|z|^j}{2}=\dfrac{|z|}{2(1-|z|)},\]
since $|d_j-\gamma-1/2|\le 1/2$.
Thus, \(z\) is inside the solution space of
\[\left|\dfrac{2y-(2y-2\gamma-1)Z}{1-Z}\right|\leq\dfrac{|Z|}{1-|Z|}.\]

For \(R\in \mathbb{R}^+\), we consider the following circles, which are both symmetric with respect to the real axis:
\[\mathscr{A}_R:=\left\lbrace Z \in \mathbb{C} : R= \left|\dfrac{2y-(2y-2\gamma-1)Z}{1-Z}\right|\right\rbrace\]
and 
$$\mathscr{C}_R:=\left\lbrace Z \in \mathbb{C} :  R=\dfrac{|Z|}{1-|Z|}\right\rbrace. $$ 
The circles $\mathscr{A}_R$ and $\mathscr{C}_R$ intersect at at most 2 distinct points. Figure \ref{fig:arbetaplot} illustrates the plots of $\mathscr{A}_R$ and $\mathscr{C}_R$ when $\gamma=0$ and $ y=1$ for various values of $R$.

\begin{figure}[htbp] 
     \centering
     \begin{subfigure}[b]{0.25\textwidth}
         \centering
         \includegraphics[width=\textwidth]{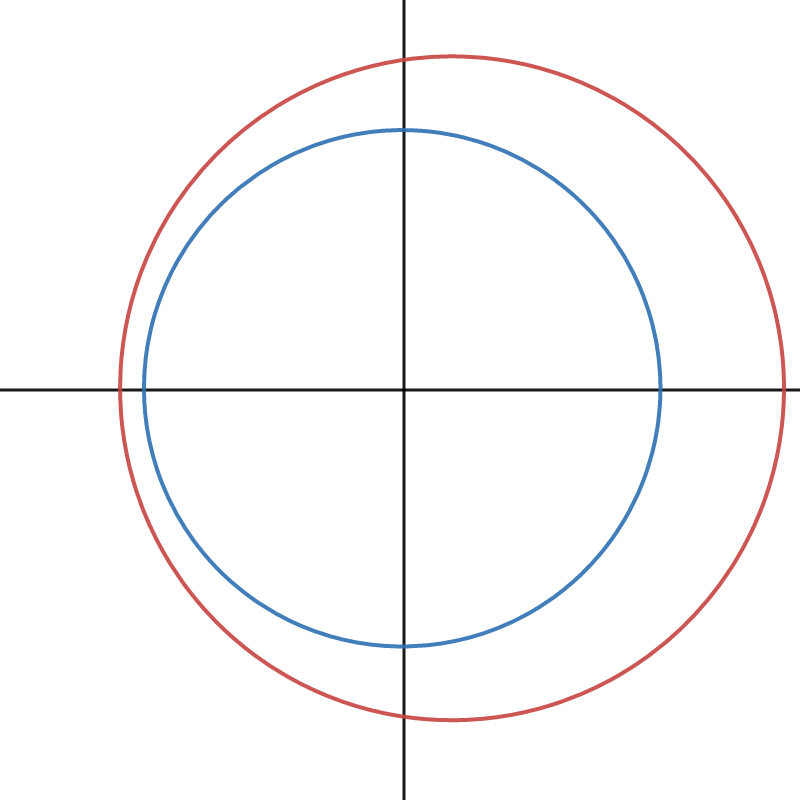}
         \caption{$R=1.775$}
     \end{subfigure}
     \hfill
     \begin{subfigure}[b]{0.25\textwidth}
         \centering
         \includegraphics[width=\textwidth]{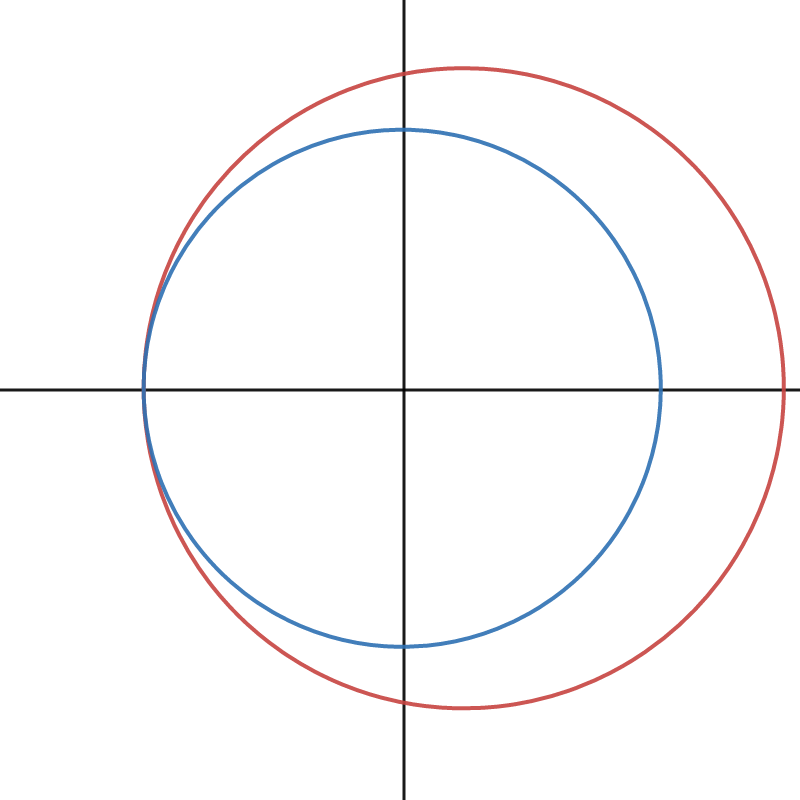}
         \caption{$R=\tfrac{3+\sqrt{17}}{4}$}
     \end{subfigure}
     \hfill
     \begin{subfigure}[b]{0.25\textwidth}
         \centering
         \includegraphics[width=\textwidth]{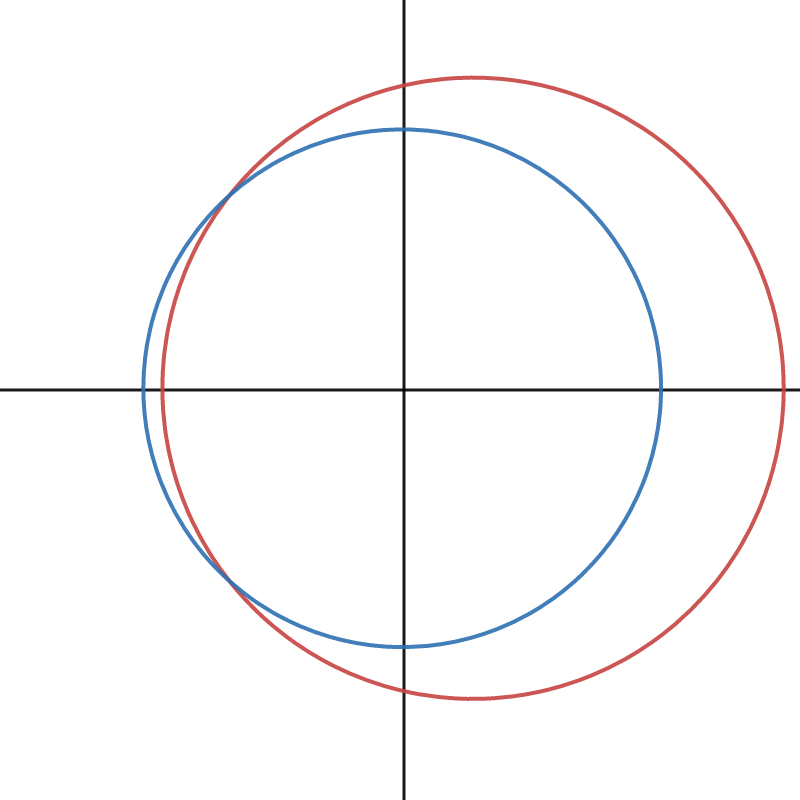}
         \caption{$R=1.8$}
     \end{subfigure}
         \caption{Plot of $\mathscr{A}_R$ (red) and $\mathscr{C}_R$ (blue) for each given $R$}
        \label{fig:arbetaplot}
\end{figure}

So, $|z|\geq|z(\gamma;y)|$ where $z(\gamma;y)$ is the real solution to 
\begin{equation} \label{equationforbound}
\dfrac{|Z|}{1-|Z|}=\left|\dfrac{2y-(2y-2\gamma-1)Z}{1-Z}\right| 
\end{equation}
with minimal absolute value.

We now state the following theorem. 
\begin{theorem}\label{bounds}
Let $x\in[\gamma,\gamma+1]$ and $\beta_1,\dots,\beta_n\in\R$ such that $\lim_{m\to\infty}|B[m]|=\infty$ where $B=(\overline{\beta_1,\dots,\beta_n})$. 
For $i\in\{1,\dots,n\}$ such that $\prod_{j=1,j\neq i}^n\beta_i$ is fixed by the map $\beta_i\to \lambda$ where $\lambda$ is a nontrivial conjugate of $\beta_i$ over $\Q(\beta_{i+1},\dots,\beta_n,\beta_1,\dots,\beta_{i-1})$, if $y := T^{i-1}(x) \neq 0$,   then we have the following.
\begin{enumerate}

\item [1.] If \(y\neq \gamma, \gamma+1\), then 
 $$|z(\gamma;y)|=\begin{cases} 
 \dfrac{\gamma+1}{2(\gamma-y)} + \dfrac{1}{2} \sqrt{\dfrac{(\gamma+1)^2+4y(y-\gamma)}{(\gamma-y)^2}},&\text{ if }\gamma\geq-\tfrac{1}{2}\text{ and }y\geq0\\
\frac{y}{y-\gamma},&\text{ if }-1\leq\gamma<-\tfrac{1}{2}\text{ and }y\geq0\\
\frac{y}{y-\gamma-1},&\text{ if }-\tfrac{1}{2}\leq\gamma<y<0\\ 
\frac{\gamma}{2(\gamma+1-y)}+\frac{1}{2}\sqrt{\frac{\gamma^2-4y(\gamma+1-y)}{(\gamma+1-y)^2}},&\text{ if } \gamma\leq -\tfrac{1}{2}\text{ and } y\leq 0.\\ 
\end{cases}$$

\item [2.] If $y=\gamma$, then
$$|z(\gamma;y)|=
    \begin{cases}
    \frac{y}{y+1} & \text{ if } y > 0 \\
     - y  &\text{ if }  0 > y \ge -1/2 \\
     \frac{y+\sqrt{y^2-4y}}{2} &\text{ if } -1/2 \ge y.
     \end{cases}$$
\item [3.] If $y=\gamma+1$, 
then $$|z(\gamma;y)|=
    \begin{cases}
    - \dfrac{y -\sqrt{y^2+4y}}{2} & \text{ if } y \geq 1/2\\
    y & \text{ if } 0<y \leq 1/2\\
    \dfrac{y}{y-1} &\text{ if } y <0.
    \end{cases}$$ 
\end{enumerate}   
\end{theorem}

\begin{proof}
    We only show the case when $y \neq \gamma, \gamma+1$. We find a real solution $z$ with $0<|z|<1$ for Equation (\ref{equationforbound}). Consider the following cases:

    \begin{itemize}
        \item[(1)] Suppose $0<z<1$ and $2y-(2y-2\gamma-1)z>0$. Then $z=\frac{y}{y-\gamma}$. Since $y-\gamma>0$, then $y>0$. Also, $y<y-\gamma$ and thus, $\gamma<0$. We see that $\frac{y}{y-\gamma}$ is a real solution if and only if $y>0$ and $\gamma<0$.

        \item[(2)] Suppose $0<z<1$ and $2y-(2y-2\gamma-1)z<0$. Using similar computations as in (1), we have that $\frac{y}{y-\gamma-1}$ is a real solution if and only if $y<0$ and $\gamma>-1$.

        \item[(3)] Suppose $-1<z<0$ and $2y-(2y-2\gamma-1)z>0$. Equation (\ref{equationforbound}) yields
        $$(\gamma - y)z^2+(\gamma+1)z+y=0.$$ Then $$z=V_\pm=\frac{-\gamma-1\pm\sqrt{(\gamma+1)^2+4y(y-\gamma)}}{2(\gamma-y)}.$$ 
        Note that $V_\pm$ is real if and only if $4y^2-4\gamma y+(\gamma+1)^2 \geq 0$. 
        The discriminant of  $4X^2-4\gamma X +(\gamma+1)^2$ is $16(-2\gamma-1)$. So, $V_\pm$ is real if $\gamma \geq -\tfrac{1}{2}$. On the other hand, if $\gamma < -\tfrac{1}{2}$, then $V_\pm$ is real if  either $y\leq (\gamma - \sqrt{-2\gamma -1})/2$ or $y\geq (\gamma + \sqrt{-2\gamma -1})/2$. 

        \item[(4)] Suppose $-1<z<0$ and $(2y-2\gamma-1)z-2y>0$. Equation (\ref{equationforbound}) yields $$(\gamma+1-y)z^2+\gamma z+y=0.$$ 
        So, $$z=W_{\pm}:= \frac{-\gamma\pm \sqrt{\gamma^2-4y(\gamma+1-y)}}{2(\gamma+1-y)}.$$
        Note that $W_{\pm}$ is real if and only if either
        \begin{enumerate}
            \item[i.] $\gamma\leq-\tfrac{1}{2}$; or 
            \item[ii.] $\gamma>-\tfrac{1}{2}$, and $y\le (\gamma+1 - \sqrt{2\gamma+1})/2$ or $y\ge (\gamma+1 + \sqrt{2\gamma+1})/2$.
        \end{enumerate}
    \end{itemize}

    The result follows from the comparison of the absolute values of the obtained real solutions. 

    Suppose $\gamma\geq0$ and $y>0$. We need to consider Cases 3 and 4. Note that $V_+<0$ while $V_->0$. Then $V_-$ is not a solution. 
    We can show $2y-(2y-2\gamma-1)V_+\geq 0$. Meanwhile, $2y-(2y-2\gamma-1)W_\pm\geq 0$. Thus, $W_\pm$ is not a solution while $V_+$ is. So, the only possible real solution is $V_+$.

    Next, suppose $-1\leq \gamma\leq 0$ and $y\geq 0$. We need to consider Cases 1, 3 and 4. Note that $V_-,W_\pm\geq 0$ and so, $z$ cannot be $V_-$ or $W_\pm$. Meanwhile, $V_+\leq 0$ and $2y-(2y-2\gamma-1)V_+\geq0$. So, the possible real solutions are $V_+$ and $\frac{y}{y-\gamma}$ (from Case 1). Observe that $|V_+|\leq \frac{y}{y-\gamma}$ if and only if $\gamma\geq -\frac{1}{2}$.

    Now, suppose $-1\leq \gamma\leq 0$ and $y<0$. The relevant cases are Cases 2, 3 and 4. Now, $V_\pm,W_+\geq0$. Hence, $V_\pm$ and $W_+$ are not solutions. Moreover,  $W_-\leq 0$ and $2y-(2y-2\gamma-1)W_-\leq 0$. So, the possible real solutions are $W_-$ and $\frac{y}{y-\gamma-1}$. Observe that $|W_-|\geq \left|\frac{y}{y-\gamma-1}\right|$ if and only if $\gamma\geq -\frac{1}{2}$.

    Finally, suppose $-1>\gamma$. Note that $y<\gamma+1<0$. So, we only consider Cases 3 and 4. Now,  $2y-(2y-2\gamma-1)V_\pm\leq 0$. Hence, $V_\pm$ is not a solution. Also, $W_+>0$ and so, $W_+$ is not a solution. So, the only possible solution is $W_-$. 
    \end{proof}

As a consequence, we have the following.
\begin{corollary}[Main Result]  \label{boundforlambda} Let $x\in[\gamma,\gamma+1]$ and $\beta_1,\dots,\beta_n\in\R$ such that $\lim_{m\to\infty}|B[m]|=\infty$ where $B=(\overline{\beta_1,\dots,\beta_n})$. Fix $i\in\{1,\dots,n\}$ and let $y = T^{i-1}(x)$. Suppose $y\neq0$. Let $\lambda$ be a nontrivial conjugate of $\beta_i$ over $\Q(\beta_1,\dots,\beta_{i-1},\beta_{i+1},\dots,\beta_n$).
Suppose that $M_i =B[n]/\beta_i$ is fixed by the conjugate map $\beta_i \mapsto \lambda$.
Then, $$|\lambda|\leq \dfrac{1}{|M_iz(\gamma;y)|}.$$ 
 In particular, if $\gamma=0$ and $y=1$, then $|z(0;1)| =  1/\varphi$ where $\varphi =(1+\sqrt{5})/2$ and so,
\[|\lambda| \le \frac{\varphi}{|M_i|}.\]
\end{corollary}
\begin{remark}
We can show that the bound $\varphi/|M_i|$ is sharp by considering the series $f(z)=1+\sum_{j=1}^\infty z^{2j-1}$ (which is the same function used in \cite[Corollary 2.3]{Solomyak1994}) and using Theorem \ref{equality} provided in the next section.
\end{remark}

\section{The $(\alpha,\beta)$-expansion}

Suppose $B=(\beta_1,\dots,\beta_m,\overline{\beta_{m+1},\dots,\beta_{m+n}})$. Each sequence $(d_1,d_2,\dots) \in \mathcal{A}(B)$ that represents a $B$-expansion $d(B;x)$ for some $x\in [0,1)$ is said to be admissible and, under certain conditions, can be described via the expansions $d(\sigma^k(B);\gamma)$ and $d^*(\sigma^k(B);\gamma+1)$ where 
$\sigma^k(B):=(\beta_{k+1},\beta_{k+2},\dots)$ for $k\in \mathbb{N}$ \cite{Dem2020, charlier2021expansions}.

Observe that $(d_{m+1},d_{m+2},\dots)$ is $(\overline{\beta_{m+1},\dots,\beta_{m+n}})$-admissible if $(d_1,d_2,\dots)$ is $B$-admissible. 
We have following the generalization of Parry numbers.

\begin{definition} \label{parrytuple}
A real base $B=(\beta_1,\beta_2,\dots)$ is a \textbf{Parry Sequence} if both $d(B_i;\gamma)$ and $d^*(B_i,\gamma+1)$ are (eventually) periodic for all $i\in\N$ where $B_i=(\beta_i,\beta_{i+1},\dots)$.
\end{definition}

 Let $B=(\overline{\alpha,\beta})$ where $\alpha, \beta \in \mathbb{R}$ such that $|\alpha\beta|>1$. To proceed, we first give a restatement of Proposition 3.4 found in \cite{Dem2020} for $\gamma=0$.

  \begin{proposition}\label{5cases}
            Let $\alpha,\beta>0$ such that $\alpha\beta>1$. 
            Let $d(\alpha,\beta;1)=(c_1,c_2,\dots)$ 
            and $d(\beta, \alpha;1) =(q_1,q_2,\dots)$.
            \begin{enumerate}
                \item[1.] [\textbf{Type 1}] If $d(\alpha, \beta;1)$ is not finite (i.e., it does not end in an infinite string of zeroes), then $d^*(\alpha, \beta;1)=d(\alpha, \beta;1)$. 
                
                \item[2.] Suppose $d(\alpha, \beta;1) = (c_1, \dots, c_m, \overline{0})$ is finite with $c_m \neq 0$. 
                    \begin{enumerate}
                        \item[a.] [\textbf{Type 2}] If $m$ is even, then
                        \[d^*(\alpha, \beta;1)=(\overline{c_1,c_2,\dots,c_{m-1},c_m-1}).\]
                        
                        \item[b.] Suppose $m$ is odd.
                        \begin{enumerate}
                            \item[i.] Suppose  $d(\beta, \alpha; 1)=(q_1, \dots, q_{m'}, \overline{0})$ is finite with $q_{m'}\neq 0$. 
                            \begin{itemize}
                                \item[] [\textbf{Type 3}] If $m'$ is even, then 
                                \[d^*(\alpha, \beta;1)=(c_1,c_2,\dots,c_{m-1},c_m-1,\overline{q_1,q_2,\dots,q_{m'-1},q_{m'}-1}).\] 
                                
                                \item[] [\textbf{Type 4}] If $m'$ is odd, then 
                                \[d^*(\alpha, \beta;1)=(\overline{c_1,c_2,\dots,c_{m-1},c_m-1,q_1,q_2,\dots,q_{m'-1},q_{m'}-1}).\]
                        \item[ii.] [\textbf{Type 5}] If $d( \beta, \alpha;1)$ is not finite, then 
                        \[d^*(\alpha, \beta;1)=(c_1,c_2,\dots,c_{m-1},c_m-1,q_1,q_2,\dots).\] 
                            
                                \end{itemize}
                               
                        \end{enumerate}
                    \end{enumerate}
            \end{enumerate}
        \end{proposition}

 \begin{remark}
  For an ordered pair $(\alpha, \beta)$ of Type 1, 2, 3 or 4, it follows that
     $d(\alpha, \beta; 1)$ is eventually periodic if and only if $d^*(\alpha, \beta; 1)$ is eventually periodic. Now, suppose $(\alpha, \beta)$ is of Type 5.
     Then $d^*(\alpha, \beta; 1)$ is eventually periodic if and only if $d(\beta, \alpha;1)$ is eventually periodic. 
     Note that if  $(\alpha, \beta)$ is of Type 5, then $(\beta,\alpha)$ is of Type 1. By Proposition \ref{5cases} (1), we have $d^*(\alpha, \beta; 1)$ is eventually periodic if and only if $d^*(\beta, \alpha; 1)$ is eventually periodic.

     \end{remark}

 \begin{proposition} \label{equiv} Let $\gamma=0$ and $\alpha,\beta>0$ such that $\alpha\beta>1$.
 The following are equivalent:
 \begin{enumerate}
     \item The expansions $d(\alpha,\beta;1)$ and $d(\beta,\alpha;1)$ are eventually periodic.
     \item The expansions $d^*(\alpha,\beta;1)$ and $d^*(\beta,\alpha;1)$ are eventually periodic.
 \end{enumerate}
 \end{proposition}
 
 \begin{proof}
 
      Note that $d(\alpha,\beta;\gamma+1)=d^*(\alpha,\beta;\gamma+1)$ if and only if $d(\alpha,\beta;\gamma+1)$ is not finite (i.e., $(\alpha, \beta)$ is of Type 1) by Proposition \ref{5cases}. So, $(2) \implies (1)$ holds since, if say $d(\alpha,\beta;\gamma+1)$ is not eventually periodic, then it must not be finite and consequently, $d(\alpha,\beta;\gamma+1)=d^*(\alpha,\beta;\gamma+1)$.
     
      The implication  $(1) \implies (2)$ readily follows from the remark following Proposition \ref{5cases}.
 \end{proof}
 
 \begin{remark}
 The above proposition, in fact, holds even when one of $\alpha$ and $\beta$ is negative (and $|\alpha \beta|>1$). 
     Similar results on the determination of the form of $d^*(\alpha, \beta;1)$ based on $d(\alpha, \beta; 1)$ and $d(\beta, \alpha; 1)$ when one of $\alpha$ and $\beta$ is negative can be found in \cite[Appendix C]{Nollen2021}. 
 \end{remark}

Now, consider the case where $\gamma=0$, $\alpha = (3+\sqrt{21})/6$ and $\beta=\sqrt{3}$. Then $d(\alpha,\beta;1) = (1,0,0,1,\overline{0})$ and $d^*(\alpha,\beta;1)=(\overline{1,0,0,0})$. Meanwhile, the minimal polynomial of $\sqrt{3}$ over $\mathbb{Q}(\alpha)$ is $x^2-3$. 
Since $\sqrt{3}> \varphi / \alpha$, then any of $d(\beta,\alpha;1)$ and $d^*(\beta,\alpha;1)$ cannot be eventually periodic by Corollary \ref{boundforlambda}. So, $(\beta,\alpha)$ is of Type 1. By Proposition \ref{5cases}, $d(\beta,\alpha;1) = d^*(\beta,\alpha;1)$.

In other words, $d(\alpha,\beta;1)$ and $d(\beta,\alpha;1)$ (or $d^*(\alpha,\beta;1)$ and $d^*(\beta,\alpha;1)$) may not be simultaneously eventually periodic. Hence, in place of Definition \ref{parrytuple}, we introduce this broader definition.
 
 \begin{definition}
    Let $\alpha, \beta \in \mathbb{R}$ such that $|\alpha \beta|>1$, we say that the pair $(\alpha,\beta)$ is a semi-Parry pair if $d^*(\alpha,\beta;\gamma+1)$ and $d(\alpha,\beta;\gamma)$ are both eventually periodic. If $(\alpha,\beta)$ and $(\beta,\alpha)$ are both semi-Parry pairs, we say that $(\alpha,\beta)$ is a Parry pair.
 \end{definition}

We consider the following examples.

 \begin{example}
     Let $\gamma=0$. 
         Suppose either $\alpha\in(0,1]$ and $\beta>0$ such that $\alpha\beta>1$ is a Parry number. 
         Then $d^*(\alpha,\beta;1)$ is periodic.
         If in addition, $d(\alpha\beta;\beta-\lfloor\beta\rfloor)$ is periodic, then $d^*(\beta,\alpha;1)$ is also periodic which means $(\alpha,\beta)$ is a Parry pair.
         For example, let $\varphi:=(1+\sqrt{5})/2$.
         Let $\alpha=\varphi^{-1}\in(0,1]$ and $\beta=\varphi^{2}$.
         Note that $\alpha\beta=\varphi$ is a Parry number.
         Then $d^*(\alpha,\beta;1)$ is periodic.
         Indeed, $d^*(\alpha,\beta;1) = (\overline{0,1,0,0})$.
         Moreover, $\alpha\beta = \varphi$ is a Pisot number and $\beta-\lfloor\beta\rfloor = 1/\varphi\in \Q(\alpha\beta)$.
         By Theorem 3.1 of \cite{Schmidt1980}, $d(\alpha\beta;\beta-\lfloor\beta\rfloor)$ is periodic.
         Hence, $d^*(\beta,\alpha;1)$ is also periodic.
         Indeed, $d^*(\beta;\alpha;1)=(2,\overline{0,0,0,1})$.
         Therefore, $(\alpha,\beta)$ is a Parry pair.
\end{example}

\begin{example}
    Let $\gamma=0$. 
         Let $\alpha = (1+\sqrt{5})/2$ and $\beta=-\alpha$. 
         Then $d^*(\alpha,\beta,1) = (\overline{1,-1,0,-1})$ and $d^*(\beta,\alpha,1) = (-2,0,-2,d^*(\alpha,\beta,1))$. 
         So, $(\alpha,\beta)$ is a Parry pair.
 \end{example}

Suppose $(\alpha,\beta)$ is a semi-Parry pair. By Proposition \ref{algebraic}, $\alpha$ is algebraic over $\Q(\beta)$. If $\lambda$ is a nontrivial conjugate of $\alpha$ over $\Q(\beta)$, then 
$|\lambda| \le (1+\sqrt{5})/ (2 \beta)$.
 
\subsection{Structure of Conjugates of $\alpha$}

In this section, we study some algebraic properties of $\alpha$ and $\beta$ when $(\alpha,\beta)$ is semi-Parry
and $\beta$ is fixed by the conjugate map $\Q(\alpha) \longrightarrow \Q(\lambda)$ given by $\alpha \mapsto \lambda$ for any conjugate $\lambda$ of $\alpha$ over $\Q(\beta)$.

\begin{definition} Let $\gamma\in \Q(\beta)$. For $x\in [\gamma, \gamma+1]$ (with $\gamma \neq -1$  if $x = \gamma+1$), we define the 
following sets:
\begin{enumerate}
    \item $\Phi(x,\beta)$ is the set of all conjugates $\lambda$ of all $\alpha\in\R$ such that $|\lambda\beta|>1$ and  $d(\alpha,\beta;x)$ is eventually periodic.
    \item $\B(x,\beta)$ is the set of all functions 
    \[f(z)= x+\sum_{j=1}^\infty\dfrac{d_j}{\beta^j}z^j \mbox{ with } d_j\in [\gamma,\gamma+1].\]
    \item $\G(x,\beta):=\{z: |z|<|\beta| \mbox{ and } f(z)=0\text{ for some }f\in\B(x,\beta)\}$.
    \item $\G(x,\beta)^{-1}:=\{1/z:z\in\G(x,\beta)\}$.
\end{enumerate}
\end{definition}

In Theorem \ref{equality} below, we show that 
$\G(x,\beta)$ and $\overline{\Phi(x,\beta)}$ are equal when $\gamma =0$ and $\beta>1$. We also provide an analogous statement when $0<\beta<1$.

From \cite{Solomyak1994}, $\B:=\left\{f(z)= 1+\sum_{j=1}^\infty d_j z^j \mid d_j\in [0,1]\right\}$
is compact under the topology induced by seminorms $f\mapsto\max\{|f(z)|:|z|=r\}$ with $r<1$. 
Now, $\B(x,1)=\B+g(z)$ where $g(z)=x-1+\sum_{j=1}^\infty \gamma z^j$. 
Hence, $\B(x,1)$ is also compact.

Let $C(\beta)^{-1}$ be the circle centered at the origin of radius $1/|\beta|$. Then 
$\G(x,\beta)^{-1}\cup C(\beta)^{-1}$ is obtained by scaling by a factor of $1/\beta$ the set
$\G(x,1)^{-1}\cup C(1)^{-1}$, which is closed. Hence, $\G(x,\beta)^{-1}\cup C(\beta)^{-1}$ is also closed.

\begin{theorem} \label{forwardinclusion}
    Let $\lambda\in \overline{\Phi(x,\beta)}$ with $|\lambda|>1/|\beta|$. Let $\gamma\in\Q(\beta)$. Then $1/\lambda$ is a zero of a function $f\in \B(x,\beta)$; that is, 
    \[\lambda\in\G(x,\beta)^{-1}.\]
\end{theorem}

\begin{proof}
     If $\lambda\in\Phi(x,\beta)$, then  $1/\lambda\in\G(x,\beta)$ by Eq. (\ref{eq:2}). Then $\lambda\in\G(x,\beta)^{-1}$. 
     Now, suppose $\lambda$ is a boundary point of $\Phi(x,\beta)$. Then $\lambda\in \G(x,\beta)^{-1}\cup C(\beta)^{-1}$.
     However, $\lambda\notin C(\beta)^{-1}$. So, $\overline{\Phi(x,\beta)}\subseteq\G(x,\beta)^{-1}$.
\end{proof}

We give a partial converse of Theorem  \ref{forwardinclusion} by  considering the restriction $\gamma=0$ and $\beta >0$.

\begin{theorem} \label{equality}
    Let $\gamma=0$ and $\beta>0$. The following holds:
    \begin{enumerate}
        \item If $\beta\geq 1$, then $\overline{\Phi(1,\beta)}=\G(1,\beta)^{-1}$.
        \item If $0<\beta<1$, then $\overline{\Phi(1,\beta)}=\G'(1,\beta)^{-1}$
        where $$ \G'(1,\beta)^{-1} = \left\{1/\lambda: |\lambda|<\beta\text{ and } 1+\sum_{j=1}^\infty \dfrac{d_j}{\beta}\lambda^j=0, d_j\in [0,\beta]\right\}.$$
    \end{enumerate}
\end{theorem}

\subsection{Proof of Theorem \ref{equality}}
We divide the proof into two cases ---  Case 1: $\beta>1$ and Case 2: $0<\beta<1$.
\subsubsection{Case 1: $\beta>1$} \label{greaterthanone} We first consider the case where $\beta\geq 1$. Note that the forward inclusion follows from Theorem  \ref{forwardinclusion}.

 Let $\lambda \in \G(1,\beta)^{-1}$. Then  $\delta:=1/\lambda $ is a root of some function 
    \[f(z)=1+\sum_{j=1}^\infty \dfrac{d_j}{\beta^j}z^j\]
    where $d_j\in [0,1]$ and $|\delta|<|\beta|$.
 Let $0<\varepsilon< \min\{(|\beta|-|\delta|)/2,|\delta|/2 \}$ 
    such that $f(z)\neq0$ 
    whenever $0<|z-\delta|<\varepsilon$. 
    Let $K:=\min\{|f(z)|: |z-\delta|=\varepsilon\}$. 
        
    We approximate $f(z)$  by the truncated series
    \[f_n(z)=\sum_{j=0}^n \dfrac{d_j}{\beta^j}z^j\]
    where $d_0=1$ and $n\in\N$ is chosen such that
    \[|f(z)-f_n(z)|<K/3\]
    for $|z-\delta|<|\beta|$.
    
        In turn, we approximate $f_n(z)$ by a function 
        \[g(z)=\sum_{j=0}^n\dfrac{M_j}{Q\beta^j}z^j\]
        where $Q\in \mathbb{N}$, $M_j \in \mathbb{N} \cup \{0\}$ with $M_0=Q$ such that 
        the rational coefficient $M_j/Q \in (0,1)$ approximates $d_j$ for $j \neq 0$ and $|g(z)-f_n(z)|<K/3$.
        Since $M_j/Q \neq 0,1$ for $j \neq 0$, then $M_0 > M_j >0$.
        
        Since $\gamma=0$ and $\beta>0$, then $0\in\mathcal{A}(\beta)$.
        Let  $N\in\N$ and  $Y$ be a variable. We consider the realization of the digit sequence $(NM_0,0,NM_1,0,\dots,NM_n,\overline{0})$ as a $B$-expansion where 
        $B=(\overline{Y, \beta})$, and suppose that this is the $B$-expansion of 1, i.e.,
        
        \begin{align} 
        1&=\sum_{j=0}^n\dfrac{NM_j}{Y^{j+1}\beta^j}. \label{eq:3}
        \end{align}
        It follows that
        \[Y^{n+1}-N\sum_{j=0}^n \dfrac{M_j}{\beta^j}Y^{n-j}=0. \]
        Let
        \[ q(Y):=Y^{n+1}\]
        and
        \[h(Y):=\sum_{j=0}^n \dfrac{M_j}{\beta^j}Y^{n-j}.\]
        The solutions to Eq. (\ref{eq:3}) are precisely the roots of 
        \[F_N(Y):=q(Y) - N h(Y).\]
        
        Since $\lim_{Y\rightarrow\infty} h(Y)=\infty$, 
        we can find $r>1$ such that $h(r)>0$. 
        We can choose $N$ large enough such that $F_N(r)<0$. Since $\lim_{Y\rightarrow\infty} F_N(Y)=\infty$, then  $F_N$ has a root $\theta$ (dependent on the choice of $N$ ) greater than $r$.

        \vspace{0.5cm}
        \noindent 
        \textbf{Claim 1.} $d(\theta, \beta;1)=(NM_0,0,NM_1,0,\dots,NM_n,\overline{0})$
        
        \noindent 
        \textbf{Proof of Claim 1.} We need to show that $NM_j \in \mathcal{A}(\theta)$, i.e.,
        $0 < \theta - NM_j<1$.
       From Eq. (\ref{eq:3}), we get
        \[1= \frac{NM_0}{\theta}+ \frac{1}{\theta}\sum_{j=1}^n \frac{NM_j}{\theta^j\beta^j}.\]
        Hence,
        \[\theta-NM_j\geq \theta-NM_0 =  \sum_{j=1}^n \frac{NM_j}{\theta^j\beta^j} > 0.\]
        It follows that $0<NM_j<\theta$ for all $j\ge 0$ and, in this case, $NM_j \in \mathcal{A}(\theta)$.  
        
    Finally, observe that 
    \begin{align*}
        \theta -NM_0&=\sum_{j=1}^{n}\dfrac{NM_j}{\theta^{j}\beta^j}\leq \sum_{j=1}^{\infty}\dfrac{NM_0-1}{\theta^{j}\beta^j}\\
        &= \dfrac{NM_0-1}{\theta\beta-1}<\dfrac{NM_0-1}{\theta-1}<\dfrac{\theta-1}{\theta-1}=1.
    \end{align*}
    In summary, we have
      \begin{equation}
        NM_0\leq\theta<NM_0+1.\label{eq:4}
    \end{equation}
    
    Hence, by the admissibility criterion provided in \cite[Theorem~4.12]{Dem2020}, the $(\theta,\beta)$-expansion of $1$ is
    \[d(\theta, \beta;1)=(NM_0,0,NM_1,0,\dots,NM_n,\overline{0}).\]
    Indeed, every odd-power shift of $(NM_0,0,NM_1,0,\dots,NM_n,\overline{0})$ is lexicographically less than $d(\theta, \beta;1)$. Moreover, the first digit of  $d^*(\beta,\theta;1)$ is $\lfloor\beta\rfloor\geq1$ since $\beta>1$. 
    Hence, every even-power shift $(0,NM_j,\dots)$ of $(NM_0,0,NM_1,0,\dots,NM_n,\overline{0})$ is lexicographically less than $d^*(\beta,\theta;1)$.
 
   \vspace{0.5cm}
   
   \noindent
   \textbf{Claim 2.} $F_N(Y)=Y^{n+1}-N\sum_{j=0}^n \frac{M_j}{\beta^j}Y^{n-j}$ is irreducible over $\Q(\beta)$ for infinitely many $N\in\N$.
   
   \noindent
   \textbf{Proof of Claim 2.}
   
  We divide the proof of this claim into two cases: $\beta$ is algebraic and $\beta$ is transcendental over $\mathbb{Q}$. To proceed, we first develop an irreducibility criterion that we can apply to  polynomials over  $\mathbb{Q}(\beta)$.
   
  \begin{lemma}[Eisenstein's Irreducibility Criterion, {\cite[Chap. 3, Theorem~6.15]{hungerford_algebra_2003}}] \label{eisen}
      Let $D$ be a unique factorization domain.
      If $f =\sum_{i=0}^n a_ix^i\in D[x]$ with $\deg(f)\geq 1$ and $p$ is an irreducible element of $D$ such that 
      \[ p \nmid a_n;\quad \quad  p\mid  a_i \text{ for }i=0,1,\dots,n-1;\quad\quad   p^2 \nmid a_0, \]
      then $f$ is irreducible over the field $\mathrm{Frac}(D)$ of fractions of $D$. If, in addition, $f$ is monic, then $f$ is irreducible over $D$.
  \end{lemma}
  
   Now, suppose that $\beta$ is transcendental over $\mathbb{Q}$. Note that $\Q(\beta)=\mathrm{Frac}(\Z[\beta])$ and $\Z[\beta]$ is a UFD. 
   Take $N$ to be a prime larger than $M_n$. By Lemma \ref{eisen}, $F_N$ is irreducible over $\Q(\beta)$ since the rational primes are also primes in $\Z[\beta]$. 
   
  Next, suppose that $\beta$ is algebraic over $\mathbb{Q}$. Then $\Q(\beta)$ is a number field. Denote by $\mathcal{O}_{\Q(\beta)}$ its ring of integers.
  
  Note that $F_N(Y)$ is irreducible over $\Q(\beta)$ 
  if and only if
  \[\beta^n  F_N(Y/\beta)= \dfrac{Y^{n+1}}{\beta}-N\sum_{j=0}^n M_jY^{n-j}\]
  is irreducible over $\Q(\beta)$.
  Let $\beta = q_1/q_2$ where $q_1,q_2\in \mathcal{O}_{\Q(\beta)}$. 
  Then $F_N(Y)$ is irreducible over $\Q(\beta)$
  if and only if 
  \[t_N(Y):=q_1\beta^n  F_N(Y) = q_2Y^{n+1}-Nq_1\sum_{j=0}^n M_jY^{n-j}\]
  is irreducible over $\Q(\beta)$. 
  
  We want to apply Lemma \ref{eisen} to $t_N\in \mathcal{O}_{\Q(\beta)}[Y]$. However, $\mathcal{O}_{\Q(\beta)}$ is not necessarily a UFD (e.g. when $\beta=\sqrt{5}$). To fix this, we embed $t_N$ to 
  the localization $(\mathcal{O}_{\Q(\beta)})_\mathfrak{p}$ where $\mathfrak{p}$ is a prime ideal of $\mathcal{O}_{\Q(\beta)}$.
  Note that $(\mathcal{O}_{\Q(\beta)})_\mathfrak{p}$ is an intermediate ring of $\mathcal{O}_{\Q(\beta)}$ and $\Q(\beta)$.

  \begin{lemma}[{\cite[Chapter~1 Proposition~11.2]{Neukirch1999}}]\label{pidlemma}
      Let $\mathfrak{p}$ be a prime ideal of a Dedekind domain $\mathcal{O}$. Then the localization $\mathcal{O}_\mathfrak{p}:=\{a/b: a,b\in \mathcal{O}, b\notin\mathfrak{p}\}$ is a principal ideal domain and hence a UFD.
  \end{lemma}
  
  Let $\mathfrak{p}$ be a prime ideal of Dedekind domain $\mathcal{O}$. Observe that $\mathcal{O}\subseteq \mathcal{O}_\mathfrak{p}\subseteq \mathrm{Frac}(\mathcal{O}) $. 
  Since $\mathrm{Frac}(\mathcal{O}_\mathfrak{p})$ is the smallest field containing $\mathcal{O}_\mathfrak{p}$, then $ \mathrm{Frac}(\mathcal{O}) \supseteq \mathrm{Frac}(\mathcal{O}_\mathfrak{p})$.
  Moreover, $\mathcal{O}\subseteq \mathcal{O}_\mathfrak{p}\subseteq \mathrm{Frac}(\mathcal{O}_\mathfrak{p})$. 
  Since $\mathrm{Frac}(\mathcal{O})$ is the smallest field that contains $\mathcal{O}$, then $\mathrm{Frac}(\mathcal{O})\subseteq \mathrm{Frac}(\mathcal{O}_\mathfrak{p})$. Thus, $\mathrm{Frac}(\mathcal{O})= \mathrm{Frac}(\mathcal{O}_\mathfrak{p})$. 
  This allows us to extend Lemma \ref{eisen} to Dedekind domains such as $\mathcal{O}_{\Q(\beta)}$.
  
  \begin{lemma}[Extended Eisenstein's Criterion for Dedekind Domains] \label{exteisen}
   Let $\mathcal{O}$ be a Dedekind domain. If $f =\sum_{i=0}^n a_ix^i\in D[x]$ with $\deg(f)\geq 1$ and $\mathfrak{p}$ is a prime ideal of $\mathcal{O}$ such that 
   \[ a_n\notin\mathfrak{p}; \quad \quad   a_i\in\mathfrak{p} \text{ for }i=0,1,\dots,n-1;\quad\quad  a_0\notin\mathfrak{p}^2,\]
    then $f$ is irreducible over the field $\mathrm{Frac}(\mathcal{O})$ of fractions of $\mathcal{O}$.
     
  \end{lemma}
  
  \begin{proof}
      By Lemma \ref{pidlemma}, $\mathcal{O}_\mathfrak{p}$ is a PID. Let $p$ be the generator of the prime ideal $\mathcal{O}_\mathfrak{p}\cap\mathfrak{p}$ of $\mathcal{O}_\mathfrak{p}$. Then $p$ is a prime element of $\mathcal{O}_\mathfrak{p}$ which is irreducible since $\mathcal{O}_\mathfrak{p}$ is a PID. Viewing $f\in \mathcal{O}_\mathfrak{p}[x]$ and applying Lemma \ref{eisen} complete the proof.
  \end{proof}

  Next, we show the existence of infinitely many prime number $N\in \mathbb{N}$ satisfying the following properties:
  \begin{enumerate}
      \item $N$ is unramified in $\Q(\beta)$, i.e.,  $N\notin\mathfrak{q}^2$ for any prime ideal $\mathfrak{q}$ of $\mathcal{O}_{\Q(\beta)}$;
      \item for some prime ideal $\mathfrak{p}$ of $\mathcal{O}_K$, we have $N\in\mathfrak{p}$ but $q_2,q_1M_n\notin\mathfrak{p}$.
      \end{enumerate} 
      We need the following lemmas.

\begin{lemma}[{\cite[Section~4.2]{ash_course_2010}}] \label{manyunramified}
   Let $\mathbb{Q}(\beta)$ be a number field and $p\in \mathbb{N}$ be a rational prime. Then $p$ is a ramified prime in $\mathbb{Q}(\beta)$ if and only if $p$ divides the discriminant  of $\mathbb{Q}(\beta)$.
\end{lemma}

\begin{lemma}\label{niceunramified}
   Let $K$ be a number field and $\omega_1,\omega_2\in\mathcal{O}_K$. Then there are infinitely many $L\in\N$ such that there is prime ideal $\mathfrak{p}$ of $\mathcal{O}_K$ with $L\in\mathfrak{p}$ but $\omega_1,\omega_2\notin\mathfrak{p}$.
\end{lemma}

\begin{proof}
    For $j\in\{1,2\}$, we have $(\omega_j)=\mathfrak{q_{j,1}}^{e_{j,1}}\cdots\mathfrak{q}_{j,n_j}^{e_{j,n_j}}$ where $q_{j,i}$'s are distinct prime ideal of $\mathcal{O}_K$ and $e_{j,i}\in\N$. Such configuration is unique up to order of the prime ideals. Consider the set 
    \[S=\{\mathfrak{q}_{j_i}: j\in\{1,2\}, i\in\{1,\dots,r_j\}\}.\]
    Then $S$ is finite. 
    
    Now, let $N_1,N_2$ be distinct (unramified) primes of $K$ and $\mathfrak{l}$ be a prime ideal of $\mathcal{O}_K$. Then $N_1,N_2\in \Z\cap\mathfrak{l}$. Note that $\Z\cap\mathfrak{l}$ is a prime ideal of $\Z$. Since $\Z$ is a PID, then $\Z\cap\mathfrak{l}=(L)$ for some rational prime $L$. So $N_1,N_2$ are both divisible by $L$ and since $N_1,N_2$ and $L$ are rational primes, then $N_1=N_2=L$, a contradiction. Hence, any two distinct (unramified) primes are not in the same prime ideal. Hence, the mapping $L\mapsto\mathfrak{l}_L$ such that $L\in\mathfrak{l}_L$ from the set of unramified primes to prime ideals of $\mathcal{O}_{\Q(\beta)}$ is one-to-one.
    
    Consider the set $T = \{\mathfrak{l}_L: L\text{ is unramified prime in }K\text{ and } L\in\mathfrak{l}_L\}$. Then $T$ is infinite by Lemma \ref{manyunramified}. Choose $\mathfrak{l}_L\in T\setminus S$. Then $L\in\mathfrak{l}_L$ but $\omega_1,\omega_2\notin \mathfrak{l}_L$. 
    
    Let $L=L_0=\omega_{3}$. We repeat the same process to obtain unramified prime $L_1$ such that $L_1\in\mathfrak{l}_{L_1}$ and $\omega_1,\omega_2,\omega_{3}\notin\mathfrak{l}_{L_1}$. We repeat the same process and we get a sequence $L_0,L_1,\dots$ or unramified primes in $K$ such that there is a prime ideal $\mathfrak{l}_j$ of $\mathcal{O}_K$ such that $L_j\in\mathfrak{l}_j$ but $\omega_1,\omega_2\notin\mathfrak{l}_j$.
\end{proof}

   From Lemma \ref{manyunramified}, we see that there are infinitely many unramified primes in $\mathcal{O}_{\Q(\beta)}$ since there are only finitely many ramified primes. 
   Taking $\omega_1=q_2$ and $\omega_2=q_1M_n$ in Lemma \ref{niceunramified}, we can take now $N$ to be desired prime that is unramified in $\mathcal{O}_{\mathbb{Q}(\beta)}$ and take an accompanying prime ideal $\mathfrak{p}$ of $\mathcal{O}_{\mathbb{Q}(\beta)}$ such that $N\in\mathfrak{p}$ but $q_2,q_1M_n\notin\mathfrak{p}$.
   Note that $q_1NM_n\notin\mathfrak{p}^2$ since $N$ is unramified. 
   Indeed, if $\mathfrak{p}^2$ divides $(q_1NM_n)=(N)(q_1M_n)$, we get a contradiction from the fact that $\mathfrak{p}^2$ does not divide  $(N)$ and $\mathfrak{p}$ does not divide $(q_1M_n)$.
   
   By Lemma \ref{exteisen}, the polynomial $t_N$ is irreducible over $\Q(\beta)$. 
   Finally, we emphasize that there are infinitely many choices of $N$ making  $F_N$ irreducible over $\mathcal{O}_{\Q(\beta)}$ since
    there are only finitely many ramified primes
    by Lemma \ref{manyunramified} and $q_2$ and $q_1M_n$ are located in only finitely many prime ideals.
   
   \vspace{0.5cm}

    \noindent
   \textbf{Claim 3.} For sufficiently large $N$,   $F_N$ has a root arbitrarily close to $\delta$.
    
    \noindent
    \textbf{Proof of Claim 3.}
    Observe that $$\dfrac{z^{n}F_N(\tfrac{1}{z})}{NM_0\beta}=\dfrac{1}{zNM_0}-g(z).$$ Then 
    
    \begin{align*}
       \left|\dfrac{z^{n}F_N(\tfrac{1}{z})}{NM_0\beta}+g(z)\right|=\dfrac{1}{|z|NM_0}<\dfrac{1}{N|z|}.
    \end{align*}
    Taking $N$ sufficiently large, we have
    
    \begin{align*}
        \left|\dfrac{z^{n}F_N(\tfrac{1}{z})}{NM_0\beta}+g(z)\right|<\dfrac{1}{N|z|}<\dfrac{2}{N|\delta|}<\dfrac{K}{3}.
    \end{align*}
    Thus, on $|z-\delta|\leq\varepsilon$, we have
    
   \begin{align*}
       \left|\dfrac{z^{n}F_N(\tfrac{1}{z})}{NM_0\beta}+f(z)\right|
      & =\left|\left(\dfrac{z^{n}F_N(\tfrac{1}{z})}{NM_0\beta}+g(z)\right)+(f_n(z)-g(z))+(f(z)-g(z))\right|
      \\& \leq   \left|\dfrac{z^{n}F_N(\tfrac{1}{z})}{NM_0\beta}+g(z)\right|+|f_n(z)-g(z)|+|f(z)-f_n(z)|\\&<\dfrac{K}{3}+\dfrac{K}{3}+\dfrac{K}{3} = K\leq |f(z)|.
    \end{align*}
    
    Since $\delta$ is a root of $f$, by Rouch\'{e}'s theorem, $F_N(1/z)$ has a root $\zeta$ satisfying $|\zeta-\delta|<\varepsilon$. Since $\varepsilon$ is arbitrarily small, we can assume that $1/\zeta$ is not a real root of $F_N$. Thus, $1/\zeta$ is a root is a conjugate of $\theta$ over $\Q(\beta)$.
    Observe that
    \[\left|\dfrac{1}{\zeta}-\lambda\right|=\left|\dfrac{\zeta-\delta}{\zeta\delta}\right|\leq \dfrac{2\varepsilon}{|\delta|^2}.\]
 Since $\varepsilon$ is arbitrarily small, then $\lambda\approx\nicefrac{1}{\zeta}$.
Therefore, $\lambda$ is close to a conjugate of $\theta$ where $(\theta,\beta)$ is semi-Parry.

\subsubsection{Case 2: $0<\beta<1$} 
Now, we consider the case where $0<\beta<1$. 
The forward inclusion $\overline{\Phi(1,\beta)} \subseteq \G'(1,\beta)^{-1}$ follows from Eq. (\ref{eq:2}) with the additional observation that when $0<\beta<1$ and $x\in[0,1]$,
then 
\[T^{2j}(x) = \beta x-\lfloor\beta x\rfloor = \beta x -0 = \beta x\in [0,\beta].\]
In other words, $d_j$ in Eq. (\ref{eq:2})  may be taken to be inside $[0,\beta]$.

For the backward inclusion, we  follow the same strategy used in Section \ref{greaterthanone}
with some slight modifications rooting from the following: for the  even shift of $(NM_0,0,NM_1,0,\dots,NM_n,\overline{0})$ to be lexicographically less than $d^*(\beta,\theta;1)$, an analysis beyond the first digits, which are both zero, is required. 

Now, let us lay down some observations. First, the rational coefficient $M_j/M_0$ satisfies $M_j/M_0<\beta$ for $j\ge 1$ since $d_j\in[0,\beta]$.
We want to exhibit that $(NM_0,0,NM_1,0,\dots,NM_n,\overline{0})$ is $(\theta,\beta)$-admissible where $\theta$ is similarly defined as in the Section \ref{greaterthanone}.

Observe that $T(1) = \theta - \lfloor \theta\rfloor = \theta-NM_0$ by Ineq. (\ref{eq:4}). Since $0<\beta<1$, then 
\begin{align*}
T^{2j-1}(1) & = \mathbb{T}^{j-1}_{\theta\beta}(\theta-NM_0)\\
T^{2j}(1) & = \beta T^{2j-1}(1) - 0 \\
&= \beta \mathbb{T}^{j-1}_{\theta\beta}(\theta-NM_0),
\end{align*}
for $j\in \mathbb{N}$ where $\mathbb{T}_{\theta\beta}$ is the (1-dimensional) beta transformation on the interval [0,1] with the radix equal to $\theta\beta$.
 Hence, 
 \[d(\theta,\beta;1) = (NM_0,0,a_1,0,a_2,0,\dots)\]
 where $\mathfrak{d}_{\theta\beta}(\theta-NM_0)=(a_1,a_2,\dots)$.
 So, it suffices to show that 
 \[(NM_1,NM_2,\dots,NM_n,\overline{0})\]
 is $\theta\beta$-admissible. Indeed, choose $N$ large enough so that $N(M_0\beta-M_j)>1$ for all $j$.
 Then
 \[NM_j <  NM_0\beta-1\le \theta\beta - 1\]
 and 
 clearly, $\theta\beta - 1$ is less than or equal to the 
 first digit of $\mathfrak{d}_{\theta\beta}^*(1)$ (which is either $\lfloor \theta\beta\rfloor$ or $\theta\beta-1$).
 By \cite[Theorem~3]{Parry1960}, we get the desired result.

\subsection{Special Cases}

Let $\gamma=0$ and let $\alpha,\beta \in \mathbb{R}$ such that $|\alpha \beta|>1$.
If $\beta=p(\alpha)$ for some polynomial $p\in\Q[x]$, then $\alpha$ may have no nontrivial conjugates over $\Q(\beta)$, as in the case where $\beta = r\alpha$ for $0\neq r\in \mathbb{R}$. 
However, if $d(\alpha,\beta;1)$ is eventually periodic, say $(a_1,\dots,a_m,\overline{a_{m+1},\dots,a_{m+k}})$ where $m$ and $k$ are even, then $\alpha$ is an algebraic number (c.f. Eq. (\ref{eq:1})).

Let $\lambda$ be a nontrivial conjugate of $\alpha$ over $\Q$. For brevity, we write $T^j(1)$ as $T^j$. 
As in Section \ref{periodic}, we have 
            \[0=1+\sum_{j=1}^\infty\dfrac{T^{2j}}{(\lambda p(\lambda))^j}+\left(\dfrac{p(\alpha)-p(\lambda)}{\alpha-\lambda}\right)\sum_{j=1}^\infty \dfrac{\lambda T^{2j-1}}{(\lambda p(\lambda))^j}.\]
            Suppose $p(x)=rX$ where $1\leq r\in\Q$. Then $p(\alpha)-p(\lambda)=r(\alpha-\lambda)$. Hence, 
            \begin{align*}
                0&= 1+\sum_{j=1}^\infty\dfrac{T^{2j}}{( r\lambda^2)^j}+r\sum_{j=1}^\infty \dfrac{\lambda T^{2j-1}}{(r\lambda^2)^j}\\&=1+\sum_{j=1}^\infty\dfrac{T^{2j}}{ r^j\lambda^{2j}}+\sum_{j=1}^\infty \dfrac{T^{2j-1}}{r^{j-1}\lambda^{2j-1}}.
            \end{align*}
            
            In this case, if $|\lambda|>1$, then $\lambda\in\G(1,1)^{-1}=\overline{\Phi(1,1)}\setminus\overline{\mathbb{D}}$. Thus, $\lambda$ is a nontrivial Galois conjugate of some Parry number. Consequently, $|\lambda|\leq (1+\sqrt{5})/2$. The following table provides a summary of the bounds for some $p(X)$.
            
             \begin{table}[h!]
    \centering
    \begin{tabular}{|c|c|c|}
         \hline 
         $p(X)$&Field&Bound \\\hline
         $rX$ where $1\geq r\in \Q$&$\Q$&$(1+\sqrt{5})/2$ \\\hline
         $-rX$ where $r>0$&$\Q(r)$&$\min\{\kappa/r,\kappa\}$ where $\kappa^3-2\kappa^2+\kappa=1$.\\\hline
         $rX^n$ where $r>0$&$\Q(r)$&$\max\{\Upsilon(n),\Upsilon(n)/\sqrt[n+1]{r}\}$ where\\ &&$\Upsilon(n)=(\alpha^{n-1}+\sqrt{\alpha^{2n-2}+4})/2$\\\hline
         
    \end{tabular}
    \caption{Bounds on the conjugates of $\alpha$ over the given  field when $d(\alpha,\beta;1)$ is eventually periodic}
    \label{tab:tablebounds}
\end{table}

\section*{Acknowledgments}

\bibliographystyle{siam}
\bibliography{references.bib}

\end{document}